\documentclass[10pt,a4paper,reqno]{amsart}
\usepackage{amssymb,amsmath,amsthm,amstext,amsfonts}
\usepackage{amsmath,amstext,amsthm,amsfonts}
\usepackage[dvips]{graphicx}
\usepackage{epic,eepic}
\usepackage[dvips]{graphicx}
\usepackage{psfrag}
\usepackage{color}
\usepackage{amssymb}
\usepackage{upgreek}
\usepackage{mathrsfs}
\usepackage{mathabx}
\usepackage{txfonts}
\makeatletter \@addtoreset{equation}{section} \makeatother

\renewcommand\thetable{\thesection.\@arabic\c@table}

\theoremstyle{plain}
\newtheorem{maintheorem}{Theorem}

\newtheorem{maincorollary}{Corollary}
\newtheorem{theorem}{Theorem }[section]
\newtheorem{proposition}[theorem]{Proposition}
\newtheorem{lemma}[theorem]{Lemma}

\theoremstyle{definition} \theoremstyle{remark}
\newtheorem{remark}[theorem]{Remark}

\renewcommand{\epsilon}{\varepsilon}

\title[Explosion of differentiability for equivalencies between Anosov flows on $3$-manifolds]{Explosion of differentiability for equivalencies between Anosov flows on $3$-manifolds}

\author[M. Bessa]{M\'{a}rio Bessa}

\address{CMA-UBI, Departamento de Matem\'atica da Universidade da Beira Interior, Rua Marqu\^es d'\'Avila e Bolama,
  6201-001 Covilh\~a,
Portugal.}
\email{bessa@ubi.pt}

\author[S. Dias]{S\'ergio Dias}
\address{Departamento de Matem\'atica, Universidade do Porto, 
Rua do Campo Alegre, 687, 
4169-007 Porto, Portugal}

\email{sergiodias@fc.up.pt}

\author[A. A. Pinto]{Alberto A. Pinto}
\address{Departamento de Matem\'atica, Universidade do Porto, 
Rua do Campo Alegre, 687, 
4169-007 Porto, Portugal}

\email{aapinto@fc.up.pt}

\begin{document}

\maketitle

\begin{abstract}
For Anosov flows obtained by suspensions of Anosov diffeomorphisms on surfaces, we show the following type of
rigidity result: if a topological conjugacy between them is differentiable at a point, then the conjugacy has a smooth extension to the suspended $3$-manifold. These result generalize
the similar ones of Sullivan and Ferreira-Pinto for 1-dimensional expanding
dynamics and also a result of Ferreira-Pinto for 2-dimensional hyperbolic dynamics.
\end{abstract}

\noindent
{\footnotesize\textbf{Keywords:} Anosov flow,
topological and differentiable equivalence, conjugacy.

}

\section{Introduction, preliminary definitions and statement of the results}

\subsection{Introduction}

There is an established theory in hyperbolic dynamics that studies properties of the dynamics and of the topological conjugacies
  that lead to   additional regularity for the conjugacies.
In the early seventies Mostow (see \cite{Mostow}) proved that if  $\mathbb{H}/\Gamma_X$ and  $\mathbb{H}/\Gamma_Y$ are two closed hyperbolic Riemann surfaces covered by finitely generated Fuchsian groups $\Gamma_X$ and $\Gamma_Y$ of finite analytic type, and
$\phi:\overline{\mathbb{H}} \to \overline{\mathbb{H}}$ induces the isomorphism
$i(\gamma)= \phi \circ  \gamma \circ \phi^{-1}$, then $\phi$ is a M\"obius transformation if, and only if, $\phi$  is absolutely continuous.
 Later, in \cite{5SSSShhhub}, Shub and Sullivan  proved  that for any two analytic orientation
preserving circle expanding endomorphisms $f$ and $g$ of the same degree,  the conjugacy is analytic if, and only if, the conjugacy is absolutely continuous.
 Furthermore, they proved    that if  $f$ and $g$ have the same set of eigenvalues, then the conjugacy is analytic.
After these results, de la Llave \cite{12}  and Marco and Moriyon
\cite{14,15}
 proved that if Anosov diffeomorphisms have the same set of eigenvalues, then
the conjugacy is smooth. For maps with critical points,
Lyubich (see \cite{LLL1}) proved that $C^2$ unimodal maps with Fibonacci
 combinatorics and   the same eigenvalues are $C^1$ conjugate.
Later on, de Melo and  Martens \cite{MMMMM1} proved that if topological conjugate unimodal maps, whose attractors are cycles of intervals, have the same set of eigenvalues, then the conjugacy is smooth.  More recently, Dobbs (see \cite{Dobbs}) proved that if a  multimodal map $f$ has an   absolutely continuous invariant measure, with a positive Lyapunov exponent, and   $f$ is absolutely continuous conjugate to another multimodal map, then the conjugacy is  $C^r$ in the domain of some induced Markov map of $f$.

In the present paper we  study  the explosion of smoothness for topological conjugacies,
i.e.   the conditions under which the smoothness of the conjugacy in a single point extends to the whole manifold.
 Tukia, in \cite{Tukia}, extended the aforementioned result of Mostow proving that if  $\mathbb{H}/\Gamma_X$ and  $\mathbb{H}/\Gamma_Y$ are two closed hyperbolic Riemann surfaces covered by finitely generated Fuchsian groups $\Gamma_X$ and $\Gamma_Y$ of finite analytic type, and
$\phi:\overline{\mathbb{H}} \to \overline{\mathbb{H}}$ induces the isomorphism
$i(\gamma)= \phi \circ  \gamma \circ \phi^{-1}$, then $\phi$ is a M\"obius transformation if, and only if,  $\phi$ is differentiable at one radial
limit point with non-zero derivative.
Sullivan  \cite{sullivan} proved that if a topological conjugacy between analytic orientation preserving circle expanding endomorphisms of the same degree is differentiable at a  point with non-zero derivative, then
the conjugacy is analytic.
Extensions  of these results  for Markov maps and   hyperbolic basic sets on surfaces
 were developed by Faria  \cite{Edson111},  Jiang \cite{JiangRG,111jia}  and Pinto, Rand and Ferreira \cite{FP,RP}, among others.
 For maps with critical points,
Jiang \cite{4Jiang,5Jiang,6Jiang,11Jiang}  proved that quasi-hyperbolic 1-dimensional maps   are smooth conjugated in an open set with full Lebesgue measure
if the conjugacy is differentiable at a  point with  uniform bound. Very recently (see \cite{APP}), Alves, Pinheiro and Pinto proved that if a topological conjugacy between multimodal maps is $C^1$ at a point in the nearby expanding set
of f, then the conjugacy is a smooth diffeomorphism in the basin of attraction of a renormalization interval.

 In the present work we begin the generalization of these type of results for continuous-time dynamical systems by proving the corresponding result for Anosov flows obtained by suspensions of Anosov diffeomorphisms on
surfaces. More precisely, we prove that if a topological conjugacy between two Anosov flows, obtained from the suspension of Anosov maps in surfaces, is differentiable at a point then the conjugacy has a smooth extension to the suspended $3$-manifold.

\subsection{Statement of the results}

Let $M$ be a $d$-dimen\-sio\-nal closed and connected $C^\infty$ Riemannian manifold. Along this paper $d=2$ when we consider diffeomorphisms and $d=3$ when considering vector fields/flows. Any $C^1$ vector field $X\colon M\rightarrow TM$ can be integrated into a flow $X_t\colon M\rightarrow M$ which is a time-parameter group of diffeomorphisms. A flow is said to be \emph{Anosov} if the tangent bundle $TM$
splits into three continuous $DX_t$-invariant nontrivial subbundles
$E^0\oplus E^u\oplus E^s$ where $E^0$ is the flow direction, the
sub-bundle $E^s$ is uniformly contracted by $DX_{t}$ and the
sub-bundle $E^u$ is uniformly contracted by $DX_{-t}$
for all $t>0$. Of course that, for an Anosov flow, we have $\emph{Sing}(X)=\emptyset$ which follows from the fact that the dimensions of the subbundles are constant on the whole manifold. The first example was obtained studying the geodesic flow of surfaces with negative curvature (cf. \cite{A}). Anosov systems for discrete dynamical systems are defined in an analogous way and the prototypical example is given by hyperbolic linear automorphisms of tori.

Our main result is the following (see \S\ref{full} for detailed definitions):

\begin{maintheorem}\label{teo0}
Let $f: M \longrightarrow M$ and $g: N \longrightarrow N$ be two surfaces Anosov diffeomorphisms. Assume that there exists a topologically conjugacy $h: M \longrightarrow N$ between them and, moreover, $h$ is differentiable in a single point. Let $c_f$ and $c_g$ be two ceiling functions over $M$ and $N$, respectively. If $c_f$ and $c_g$ are differentiable, then the function $\hat{h}: M_{c_f} \longrightarrow N_{c_g}$ defined in (\ref{h}) is differentiable. 
\end{maintheorem}

In ~\cite{P} Plante showed that codimension-1 Anosov flows on compact and connected manifolds $M$ are known to admit global cross sections, provided the fundamental group of $M$ is solvable. Clearly, Anosov flows on 3-dimensional manifolds are codimensional.  In ~\cite{P} it is also obtained that when $M$ is a bundle over $\mathbb{S}^1$ fibered by a $2$-torus $\mathbb{T}^2$, then any Anosov flow on $M$ is topologically equivalent to the suspension of a hyperbolic automorphism on $\mathbb{T}^2$. 

As a direct consequence of ~\cite{FP} and Theorem~\ref{teo0} we obtain:

\begin{maincorollary}\label{teo1}
Let $X_t\colon M\rightarrow M$ and $Y_t\colon M\rightarrow M$ be Anosov flows on a closed $3$-manifold $M$ which is a bundle over $\mathbb{S}^1$ with fiber bundle $\mathbb{T}^2$, denote $M=\mathbb{S}^1\times \mathbb{T}^2$. If $h\colon M\rightarrow M$ is an equivalence between the two flows which is differentiable in $x\in \mathbb{T}^2$, then $h$ is differentiable in the whole $M$.
\end{maincorollary}

\section{Proof of Theorem~\ref{teo0}}\label{full}

\subsection{Suspension flows}

Let $f:M \longrightarrow M$ be a diffeomorphism and $c: M \longrightarrow \mathbb{R}^+$ a continuous function such that $c(x)\geq a > 0$ for all $x \in M$. We consider also the subspace of $M \times \mathbb{R}^+$ defined by:
$$\tilde{M}=\{(x,t) \in M \times \mathbb{R}^+: x \in M \textrm{ e } 0 \leq t \leq c(x)\}.$$
Let $M_c$ stands for the quotient space $M_c=\tilde{M}/\sim$, where $\sim$ is an equivalence relation in $\tilde{M}$ defined by $(x,c(x))\sim(f(x), 0)$. The \emph{suspension of $f$} with ceiling (or roof) function $c$ is the flow 
\begin{align}
\varphi_t:M_c &\longrightarrow M_c \nonumber \\
(x,s) &\longmapsto \big(f^n(x),s'\big) \nonumber
\end{align}
where $n$ is univocally determined by 
\begin{equation}
\sum_{i=0}^{n-1}c(f^i(x))\leq t+s < \sum_{i=0}^nc(f^i(x)) \label{formula}
\end{equation}
when $t+s \geq c(x)$; in this case we define $s'= s+t- \sum_{i=0}^{n-1} c(f^i(x))$. If $t+s < c(x)$, we take $n=0$ and $s'=s+t$. In brief words, we travel with velocity equal to one and along $\{x\} \times [0,c(x)]$, then we jump to $(f(x),0)$ and travel through $\{f(x)\} \times [0,c(f(x))]$ and so on until we spend the time $t$. We observe that the flow is well-defined because (\ref{formula}) implies $$0 \leq \underbrace{s+t- \sum_{i=0}^{n-1} c(f^i(x))}_{s'} \leq c(f^n(x)).$$

\begin{remark} Let us show that $\varphi$ defined above is a flow; $\varphi_0(x,s)=(x,s)$, since $n=0$. For the group property we have:
\begin{align*}
\varphi_t\big(\varphi_r(x,s)\big) = \; & \varphi_t\left(f^n(x), r+s-\sum_{i=0}^{n-1}c\big(f^i(x)\big)\right) =  \left(f^m\big(f^{n}(x)\big), t+r+s-\sum_{i=0}^{n-1}c\big(f^i(x)\big)-\sum_{i=0}^{m-1}c\big(f^i(f^n(x))\big)\right) \\
 = \; & \left( f^{n+m}(x), r+s+t- \sum_{i=0}^{n+m-1}c\big(f^i(x)\big)\right) 
= \varphi_{t+r}(x,s), \\
\end{align*}
where $n$ and $m$ are such that
\begin{center}
$\sum \limits_{i=0}^{n-1}c\big(f^i(x)\big)) \leq r+s < \sum \limits_{i=0}^{n}c\big(f^i(x)\big)$ \\ 

\vspace*{0.2cm}

and \quad $ \sum \limits_{i=0}^{m-1}c\big(f^i(f^n(x))\big) \leq r+s + t- \sum \limits_{i=0}^{n-1}c\big(f^i(x)\big) < \sum \limits_{i=0}^{m}c\big(f^i(f^n(x))\big)$.
\end{center}
The last inequality follows from
\begin{center}
$$\sum_{i=0}^{m-1}c\big(f^i(f^n(x))\big) \leq r+s + t- \sum_{i=0}^{n-1}c\big(f^i(x)\big) < \sum_{i=0}^{m}c\big(f^i(f^n(x))\big)$$ \\

\vspace*{-0.6cm}

$$\Downarrow$$ \\

\vspace*{-0.6cm}

$$\sum_{i=0}^{n+m-1} c\big(f^i(x)\big) \leq r+s+t < \sum_{i=0}^{n+m} c\big(f^i(x)\big).$$
\end{center}

\end{remark}

\vspace*{0.4cm}

\subsection{Topological equivalences and conjugacies}\label{TE} Two flows $\varphi: \mathbb{R} \times M \longrightarrow M$ and $\psi: \mathbb{R} \times N \longrightarrow N$ are said to be \emph{topologically equivalent} is there exists a homeomorphism $h: M \longrightarrow N$ such that $h$ sends orbits of $\varphi$ into orbits of $\psi$, and preserves the orientation. Two flows $\varphi$ and $\psi$ are said to be \emph{topologically conjugated} if there exists a homeomorphism $h: M \longrightarrow N$ sending orbits of $\varphi$ into orbits of $\psi$, preserving the orientation and also the time parametrization. Clearly, if $\varphi$ and $\psi$ are conjugated, then they are also equivalent, just take $\tau_x(t)=t$.

Let be given two diffeomorphisms $f: M \longrightarrow M$ and $g:N \longrightarrow N$. We consider the suspensions of $f$ and $g$ associated to ceiling functions $c_f: M \longrightarrow \mathbb{R}^+$ and $c_g: N \longrightarrow \mathbb{R}^+$ respectively. Let $\varphi_t$ and $\psi_t$ be the respective suspension flows. Assume that $f$ and $g$ are topologically conjugated, i.e., there exists a homeomorphism $h:M \longrightarrow N$ such that $g \circ h(x)=h \circ f(x)$ for all $ x \in M$. 

A natural question is to know if $\varphi_t$ and $\psi_t$ are still topologically conjugated. The answer, in general is negative. However, we will see that the flows are topologically equivalent. In order to prove it we must define a homeomorphism $\hat{h}: M_{c_f} \longrightarrow N_{c_g}$ which conjugates both flows and preserve the fixed orientation. We define $\hat{h}: M_{c_f} \longrightarrow N_{c_g}$ in the following way:
\begin{equation}\label{h}
\hat{h}(x,s)=\underbrace{\psi_{s\frac{c_g(h(x))}{c_f(x)}}\underbrace{h \circ \underbrace{\varphi_{-s}(x,s)}_{(x,0)}}_{(h(x),0)}}_{\left(h(x),s\frac{c_g\left(h(x)\right)}{c_f(x)}\right)}
\end{equation}
where we consider that $h: M \longrightarrow N$ extends to $h:M \times \{0\} \longrightarrow N \times \{0\}$ and we abuse and keep the same notation. In rough words, we travel by $\{x\} \times [0,c_f(x)]$ until $(x,0)$, apply $h$ and travel along the segment $\{h(x)\}\times [0, c_g(h(x))]$ the correspondent time. Given $s,t \in \mathbb{R}_0^+$, we consider the following function: 
\begin{align*}
n_{s,t}: M& \longrightarrow \mathbb{N}_0 \\
x& \longmapsto n_{s,t}(x)
\end{align*}
where $n_{s,t}(x)$ is the only integer such that
\begin{equation}
\sum_{i=0}^{n_{s,t}(x)-1}c_f\big(f^i(x)\big)\leq t+s < \sum_{i=0}^{n_{s,t}(x)}c_f\big(f^i(x)\big) \label{escolha de n}
\end{equation}
or $n_{s,t}(x)=0$, when $s+t<c_f(x)$. The map $t \mapsto n_{s,t}(x)$ is piecewise constant and increasing, for $s$ and $x$ fixed.  We would like to find $\tau_{(x,s)}: \mathbb{R} \longrightarrow \mathbb{R}$ strictly increasing such that, for all $ t \in \mathbb{R}$,
\begin{equation}
\hat{h}\big(\varphi_t(x,s)\big)= \psi_{\tau_{(x,s)}(t)}\big(\hat{h}(x,s)\big). \label{conjug}
\end{equation}
Let $t'=\tau_{(x,s)}(t)$. The equation (\ref{conjug}) is equivalent to 
\begin{eqnarray*}
&\hat{h}\left(f^{n_{s,t}(x)}(x), s+t-\sum \limits_{i=0}^{n_{s,t}(x)-1}c_f\left(f^i(x)\right)\right)=\psi_{t'}\left(h(x),s\frac{c_g\left(h(x)\right)}{c_f(x)}\right)  \\
&\Longleftrightarrow \left(\underbrace{h\left(f^{n_{s,t}(x)}(x))\right)}_{g^{n_{s,t}(x)}\left(h(x)\right)} , \left[s+t-\sum \limits_{i=0}^{n_{s,t}(x)-1}c_f\left(f^i(x)\right)\right]\frac{c_g\left(h\left(f^{n_{s,t}(x)}(x)\right)\right)}{c_f\left(f^{n_{s,t}(x)}(x)\right)}\right)=\psi_{t'}\left(h(x),s\frac{c_g\left(h(x)\right)}{c_f(x)}\right)
\end{eqnarray*}
The equality holds if and only if $t'$ is such that 
\begin{equation}
\sum_{i=0}^{n_{s,t}(x)-1}c_g\left(g^i\left(h(x)\right)\right) \leq t' + s \frac{c_g\left(h(x)\right)}{c_f(x)} <\sum_{i=0}^{n_{s,t}(x)}c_g\left(g^i\left(h(x)\right)\right) \label{escolha de n}
\end{equation}
and 
\begin{equation}
t' + s \frac{c_g\left(h(x)\right)}{c_f(x)}-\sum_{i=0}^{n_{s,t}(x)-1}c_g\left(g^i\left(h(x)\right)\right)= \left[s+t-\sum_{i=0}^{n_{s,t}(x)-1}c_f\left(f^i(x)\right)\right]\frac{c_g\left(h\left(f^{n_{s,t}(x)}(x)\right)\right)}{c_f\left(f^{n_{s,t}(x)}(x)\right)} \label{determinacao de tau}
\end{equation}

\noindent From (\ref{determinacao de tau}) it follows that 

\begin{equation*}
\tau_{(x,s)}(t)= \left[s+t-\sum_{i=0}^{n_{s,t}(x)-1}c_f\left(f^i(x)\right)\right]\frac{c_g\left(h\left(f^{n_{s,t}(x)}(x)\right)\right)}{c_f\left(f^{n_{s,t}(x)}(x)\right)} - s \frac{c_g\left(h(x)\right)}{c_f(x)}+\sum_{i=0}^{n_{s,t}(x)-1}c_g\left(g^i\left(h(x)\right)\right).
\end{equation*}

\vspace*{0.3cm}

\noindent From (\ref{escolha de n}), we get 

\begin{equation*}
0 \leq s+t - \sum_{i=0}^{n_{s,t}(x)-1}c_f\left(f^i(x)\right)< c_f\left(f^{n_{s,t}(x)}(x)\right), 
\end{equation*} 
thence
\begin{equation*}
0 \leq \frac{c_g\left(h\left(f^{n_{s,t}(x)}(x)\right)\right)}{c_f\left(f^{n_{s,t}(x)}(x)\right)} \left[ s+t - \sum_{i=0}^{n_{s,t}(x)-1}c_f\left(f^i(x)\right)\right] < c_g\left(h\left(f^{n_{s,t}(x)}(x)\right)\right)=c_g(g^{n_{s,t}(x)}\left(h(x)\right))
\end{equation*} 

\vspace*{0.3cm}

\noindent In overall,

\begin{eqnarray*}
\sum_{i=0}^{n_{s,t}(x)-1}c_g\left(g^i\left(h(x)\right)\right) &\leq& \underbrace{\frac{c_g\left(h\left(f^{n_{s,t}(x)}(x)\right)\right)}{c_f\left(f^{n_{s,t}(x)}(x)\right)} \left[ s+t - \sum_{i=0}^{n_{s,t}(x)-1}c_f\left(f^i(x)\right)\right] + \sum_{i=0}^{n_{s,t}(x)-1}c_g\left(g^i\left(h(x)\right)\right)}_{t' + s \frac{c_g\left(h(x)\right)}{c_f(x)}}\\ & <&\sum_{i=0}^{n_{s,t}(x)}c_g\left(g^i\left(h(x)\right)\right)
\end{eqnarray*}

Where the inequalities on (\ref{escolha de n}) follows. Then, $n_{s\frac{c_g\left(h(x)\right)}{c_f(x)},t'}\left(h(x)\right)=n_{s,t}(x)$
like we wanted. 

\vspace*{0.5cm}

Clearly, $\tau$ is continuous and strictly increasing because for all $(x,s) \in M_{c_f}$ we have
$$
\frac{d \tau_{(x,s)}(t)}{dt}=\frac{c_g\left(h\left(f^{n_{s,t}(x)}(x)\right)\right)}{c_f\left(f^{n_{s,t}(x)}(x)\right)} > 0.
$$

\vspace*{0.5cm}

If $\hat{h}: M_{c_f} \longrightarrow N_{c_g}$ is a homeomorphism, then $\varphi$ and $\psi$ are topologically equivalent flows. In fact:
\begin{itemize}
\item [$\cdot$]$\hat{h}$ is continuous because it is the composition of continuous functions;
\item [$\cdot$]$\hat{h}$ is invertible with inverse $$\hat{h}^{-1}(y,t):= \varphi_{t\frac{c_f\left(h^{-1}(y)\right)}{c_g(y)}} \circ h^{-1} \circ \psi_{-t} (y,t)= \left(h^{-1}(y),t\frac{c_f\left(h^{-1}(y)\right)}{c_g(y)}\right);$$
\item [$\cdot$]$\hat{h}^{-1}$ is continuous.
\end{itemize}

In conclusion we have just proved the following:
\begin{proposition}
The flows $\varphi$ and $\psi$ are topologically equivalent, being $\hat{h}$ the equivalency between them.
\end{proposition}

\begin{figure}[h]\label{fig1}
\begin{center}
  \includegraphics[width=10cm,height=6cm]{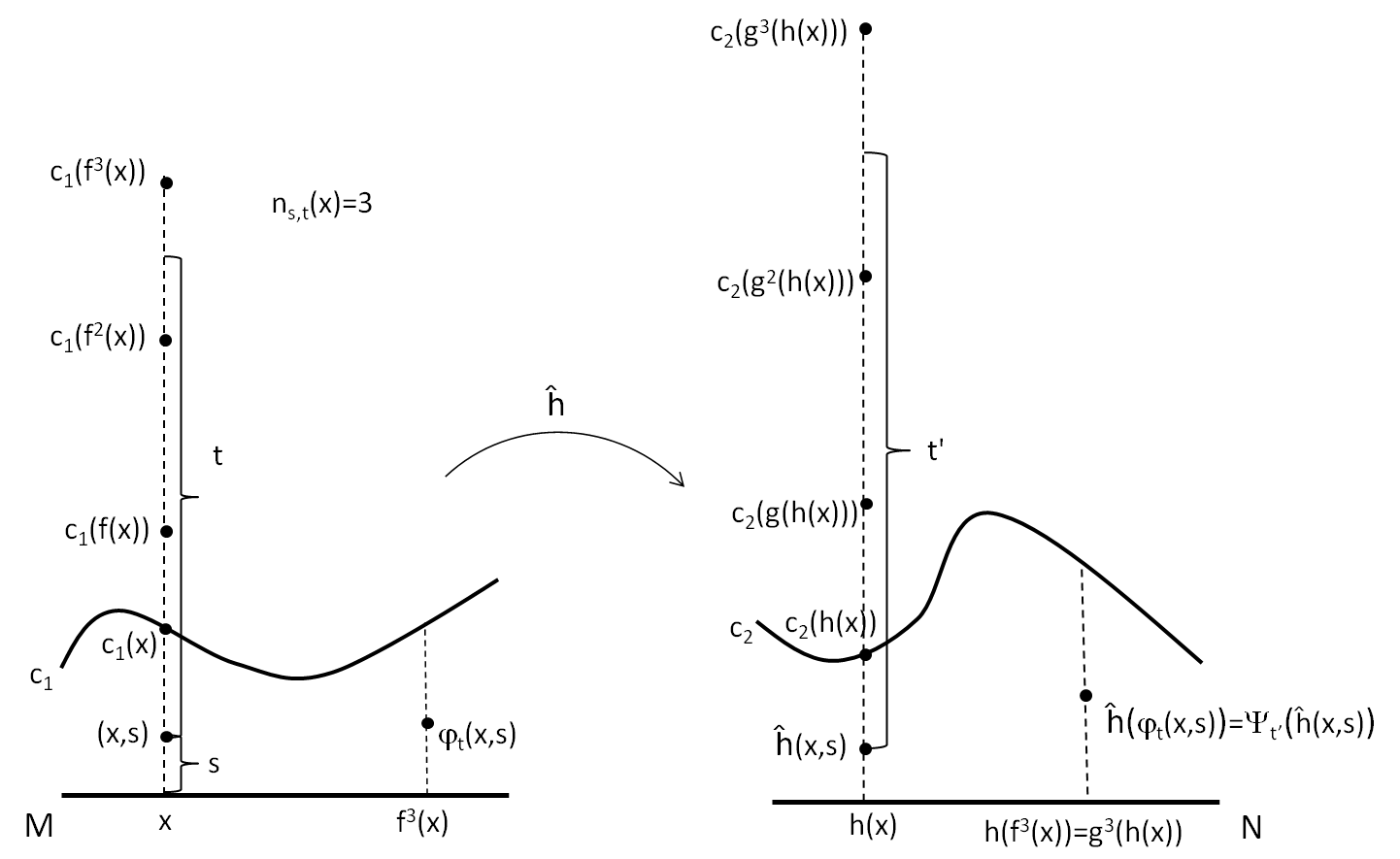}
\caption{Topologically equivalence between suspension flows.}
\label{figure}
\end{center}
\end{figure}

We begin by obtaining a preliminary result on piecewise differentiability:

\begin{proposition}
Let $f: M \longrightarrow M$ and $g: N \longrightarrow N$ be two surfaces Anosov diffeomorphisms. Assume that there exists a topologically conjugacy $h: M \longrightarrow N$ between them and, moreover, $h$ is differentiable in a single point. Let $c_f$ and $c_g$ be two ceiling functions over $M$ and $N$, respectively. If $c_f$ and $c_g$ are differentiable, then the function $\hat{h}: M_{c_f} \longrightarrow N_{c_g}$ described above is piecewise differentiable. 
\end{proposition}

\begin{proof}
The homeomorphism $\hat{h}$ is given by $$\hat{h}(x,s)=\left(h(x), s\frac{c_g\big(h(x)\big)}{c_f(x)}\right),$$ for all $(x,s) \in M_{c_f}$. Assume that $(x,s) \notin M \times \{0\}$. Then,
$D\hat{h}_{(x,s)}=\begin{pmatrix} & Dh_x & ~ & \vline\,\textbf{0} \\ \hline* & &\diamond  & \,\vline\,\star\end{pmatrix}$ where $$*=s\frac{\partial \left(\frac{c_g\left(h(x)\right)}{c_f(x)}\right)}{\partial x_1}, \qquad \diamond=s\frac{\partial \left(\frac{c_g\left(h(x)\right)}{c_f(x)}\right)}{\partial x_2} \qquad \textrm{e} \qquad \star=\frac{c_g\left(h(x)\right)}{c_f(x)},$$
being $x=(x_1, x_2)$. But, $Dh_x$ exists by the main theorem in ~\cite{FP}, $*$ and $\diamond$ exists since $c_f$ and $c_g$ are both differentiable. 

\end{proof}

The lack of differentiability in our construction lies in the way $\hat{h}$ act in the sections $M$ and $N$. Next, we carefully reparametrize $\hat{h}$ in order to achieve the differentiability of  $\hat{h}$ in the whole suspension manifold. We begin by proving a useful and elementary result about bump functions.

\begin{lemma} \label{lema2}
Given $a<b$ and $c \in \mathbb{R}$, there is a $C^\infty$ function $F_{a,b}^c: \mathbb{R} \longrightarrow \mathbb{R}$ such that $F_{a,b}^c(t)=0$ for all $t \notin (a,b)$ and $\int_{-\infty}^\infty F_{a,b}^c(s) \, ds=c$.
\end{lemma}

\begin{proof}
Consider the map 
$$ f_{a,b}(t)= \left\{ \begin{array}{ll}
e^{\frac{-1}{(a-t)(b-t)}}, \quad t \in (a,b) \\
0, \qquad \qquad \;\; \textrm{otherwise} 
\end{array} \right.$$

and now define $$F_{a,b}^c(t)=\frac{c \cdot f_{a,b}(t)}{\int_a^b f_{a,b}(s)\, ds}.$$
Then, $F_{a,b}^c$ is $C^\infty$ and, since $\int_{-\infty}^\infty F_{a,b}^c(s) \, ds=\int_a^b F_{a,b}^c(s) \, ds$, we are done.
\end{proof}


Recall that $c_f: M \longrightarrow \mathbb{R}^+_0$ and $c_g: N \longrightarrow \mathbb{R}_0^+$ are differentiable functions such that $c_f(x)$, $c_g(y) \geq \alpha >0$, for all $x \in M$ and $y \in N$. Let $\varepsilon = \alpha/3$. Fix $x \in M$ and let $\varphi_x: \mathbb{R}^+_0 \longrightarrow \mathbb{R}$ be a $C^\infty$ map such that 
\begin{itemize}
\item[i.] $\varphi_x(t)=0$ for all $t \in [0,\varepsilon];$
\item[ii.] $\varphi_x(t)=0$ for all $ t \geq c_f(x)-\varepsilon$ and
\item[iii.] $\int_\varepsilon^{c_f(x)-\varepsilon} \varphi_x(s)\,ds =c_g\big(h(x)\big)-c_f(x)$.
\end{itemize}
The existence of such function is guaranteed by Lemma \ref{lema2}.

Now, consider $\phi_x(t)=\int_0^t \;\varphi_x(s)+1 \, ds$. We have:

\begin{itemize}
\item[i.] $\phi_x(t)= \int_0^t \; 1 \,ds = t$, for all $ t \in [0,\varepsilon];$
\item[ii.] for all $t \geq c_f(x)$ we have \begin{align*}
\phi_x(t)&= \int_0^\varepsilon (\varphi_x(s)+1) \,ds +  \int_\varepsilon^{c_f(x)-\varepsilon} (\varphi_x(s)+1) \,ds+\int_{c_f(x)-\varepsilon}^t (\varphi_x(s)+1) \,ds \\
&= \varepsilon + c_g\big(h(x)\big)-c_f(x)+c_f(x)-\varepsilon - \varepsilon + t - c_f(x)+\varepsilon \\
&=c_g\big(h(x)\big)-c_f(x) + t;
\end{align*}
\item[iii.] $\phi_x \in C^\infty$.
\end{itemize}

\begin{figure}[h]\label{fig1}
\begin{center}
  \includegraphics[width=4cm,height=4cm]{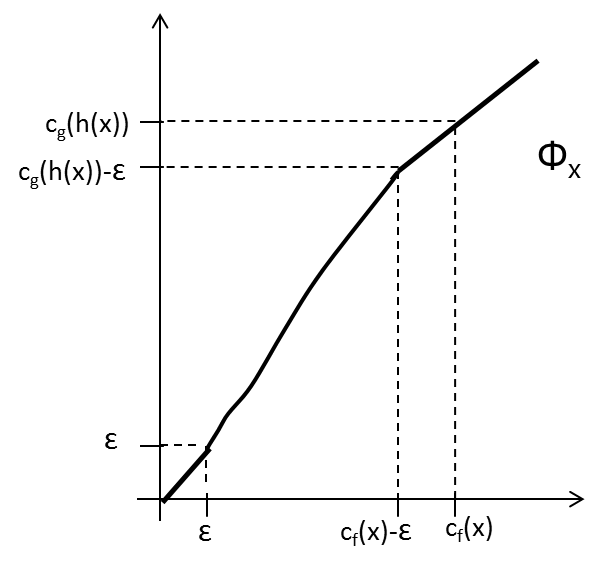}
\caption{The reparametrization $\phi_x$ on each fiber at $x$.}
\label{figure}
\end{center}
\end{figure}

Let $\hat{M}=\left\{(x,s) \in M \times \mathbb{R}: s \in [0, c_f(x)]\right\}$ and define the following equivalence relation: $\big(x,c_f(x)\big)\sim \big(f(x),0\big)$. Let $M_{c_f}=\hat{M}/ \sim$. Similarly, we define $N_{c_g}$. Now, consider $\hat{h}: M_{c_f} \longrightarrow N_{c_g}$ defined by $\hat{h}(x,s)= \big( h(x), \phi_x(s) \big)$. If $h$ is differentiable at a point, then $\hat{h}$ is differentiable.

\section*{Acknowledgements}

MB was partially supported by National Funds through FCT - ``Funda\c{c}\~{a}o para a Ci\^{e}ncia e a Tecnologia", project PEst-OE/MAT/UI0212/2011. SD and AAP thank the financial support of LIAADÐINESC TEC through Strategic Project - LA 14 - 2013-2014 with reference PEst-C/EEI/LA001 4/2013, 
Project USP-UP, IJUP, Centro de Matem\'atica da Universidade do Porto, Faculty of Sciences, University of Porto;
Calouste Gulbenkian Foundation; ERDF/FEDER-European Regional Development Fund and COMPETE Programme (operational programme for competitiveness) and by National Funds through Funda\c{c}\~ao para a Ci\^encia e a Tecnologia (FCT)
within the projects ``Dynamics and Applications" (PTDC/MAT/121107/2010), PEst-OE/MAT/UI0212/2011, COMP-01-0124-FEDER-022701
and FCOMP-01-0124-FEDER-037281.


\bigskip

\end{document}